\documentclass[amssymb,twoside,11pt]{article}
\usepackage{latexsym,amsmath,graphicx}
\usepackage{amsfonts}
\usepackage{enumitem}
\newtheorem{theorem}{Theorem}[section]
\newtheorem{lemma}[theorem]{{\bf Lemma}}

\newtheorem{ex}[theorem]{{\bf Example}}
\newtheorem{definition}{Definition}[section]
\numberwithin{equation}{section}
\newenvironment{proof}{\indent{\em Proof:}}{\quad \hfill
$\Box$\vspace*{2ex}}

\begin{document}
\setcounter{page}{1}
\begin{center}
\vspace{0.4cm} {\large{\bf On the Nonlinear Impulsive $\Psi$--Hilfer Fractional Differential Equations}}\\
\vspace{0.5cm}
Kishor D. Kucche $^{1}$ \\
kdkucche@gmail.com \\
\vspace{0.35cm}
Jyoti P. Kharade $^{2}$\\
jyoti.thorwe@gmail.com\\
\vspace{0.35cm}
 J. Vanterler da C. Sousa $^{3}$\\
ra160908@ime.unicamp.br\\
\vspace{0.35cm}
$^{1,2}$ Department of Mathematics, Shivaji University, Kolhapur-416 004, Maharashtra, India.\\
$^{3}$  Department of Applied Mathematics, Imecc-Unicamp, 13083-859, Campinas, SP, Brazil.

\end{center}
\def\baselinestretch{1.0}\small\normalsize
\begin{abstract}
In this paper, we consider the nonlinear $\Psi$-Hilfer impulsive fractional differential equation. Our main objective is to derive the formula for the solution and examine the existence and uniqueness of results. The acquired results are extended to the nonlocal $\Psi$-Hilfer impulsive fractional differential equation.  
We gave an applications to the outcomes we procured. Further, examples are provided in  support of the results we got.
\end{abstract}
\def\baselinestretch{1.5}

\allowdisplaybreaks
\section{Introduction}

The fractional differential equations (FDEs) over the years have been the object of investigation by many researchers \cite{wang22}--\cite{zhou}. The fact is that certain natural phenomena by means of fractional differential equations are modeled and allows to better describe the real situation of the problem compared to the problem modeled by means of differential equations of whole order \cite{rica}--\cite{sousa6}. Recently, Sousa et al. \cite{sousa7} presented a fractional mathematical model by means of the time-fractional diffusion equation, which describes the concentration of nutrients in the blood and allows analyzing the solution of the model, better than the integer case. In addition, other mathematical models can be obtained in the literature involving fractional differential equations \cite{dipi4}--\cite{dipi3}.

On the other hand, investigating the existence, uniqueness and stability of solutions of FDEs of the following types: functional, impulsive, evolution, with instantaneous and non-instantaneous impulses \cite{benchohra3}--\cite{wang2}. In this direction the subject  has picked up strength and interest of the researchers, since the fractional derivatives allows the variation of the order of the differential equation that is straightforwardly associated with the solution of such FDEs.

Eminent mathematicians working in the field of FDEs, has been exhibiting critical and fascinating outcomes throughout the years that contribute  significantly to the mathematical analysis of FDEs, few of them are: Balachandran ,Trujillo \cite{Krish}, Zhou \cite{zhou}, Wang \cite{wang2}, Feckan \cite{fec}, Benchohra \cite{benchohra3}, O'Regan \cite{Ravi},  Kilbas \cite{Kilbas}, JinRong Wang \cite{Jinr}, Agarwal \cite{agarwal}, Diethelm \cite{Diethelm}, Guo \cite{guo} and Mophou \cite{mophou}.

The FDEs with impulsive effect play vital role in modeling real world physical phenomena involving in the study of population
dynamics, biotechnology  and chemical technology. Advancement in the theory of impulsive  differential equations and its applications in the real world phenomena have been marvelously given in the monographs of Bainov and Simeonov \cite{bainov}, Benchohra et al \cite{BenHer} and Samoilenko and Perestyuk \cite{LakBai}.

In 2009, Benchohra and Slimani \cite{Benchohra} investigated various criterion for the existence of solutions for a class of initial value problems for impulsive fractional differential equations
given by
\begin{equation} \label{imp01}
\begin{cases}
^c D^{\mu}y(t)=f(t, y(t)),~t \in [0,T], t\neq t_{k},\\
\Delta y|_{t=t_{k}}= I_{k}(y(t_{k}^-)) \\
 y(0)=y_{0} \in \mathbb{R},
\end{cases}
\end{equation}
where $^c D^{\mu}(\cdot)$ is the Caputo fractional derivative of order $0< \mu \leq 1$, $f: [0,T] \times \mathbb{R} \rightarrow \mathbb{R}$ is a given function, $I_{k}:\mathbb{R} \rightarrow \mathbb{R}$, $k=1,...,m$ , $0=t_{0}<t_{1}<\cdots <t_{m}<t_{m+1}=T$, $\Delta y |_{t=t{k}}=y(t_{k}^{+})-y(t_{k}^{-})$, $y(t^{+}_{k})=\lim_{h \rightarrow 0^-} y(t_{k}+h)$, $k=1,2,...,m$. 

Benchohra and Seba \cite{benchohra} extended the study of existence  for  impulsive FDEs \eqref{imp01} in the Banach spaces. The following year, Benchohra and Berhoun \cite{benchohra3}, investigated sufficient conditions for the existence of solutions for impulsive FDEs  with variable times.

In  \cite{fec}  Feckan et al. with the help of the examples it is demonstrated  that the  formula for the solutions of fractional impulsive FDEs \eqref{imp01} considered  in the few referred  papers in \cite{fec} were incorrect. They have derived the valid  formula for the solution of  impulsive FDEs \eqref{imp01} involving Caputo derivative and investigated the existence results for \eqref{imp01} using Banach contraction principle and Leray-Schauder theorem.

In another interesting paper \cite{wang1}, Wang and coauthor presented the idea of piecewise continuous solutions for Caputo fractional  impulsive Cauchy problems and impulsive fractioanl boundary value problem. They acquired  existence and uniqueness of solution and furthermore determined data dependence and Ulam stabilities of solutions by means of  generalized singular Gronwall inequalities.

 It is noticed that numerous works with refined and important mathematical tools have been published and others that are yet to come \cite{sousa,sousa1,Krish,fec,Ravi,Jinr}. In any case, it is advantageous to utilize more broad fractional derivatives in which they hold a wide class of fractional derivatives as particular cases, particularly the traditional ones of Caputo and Riemann-Liouville (RL). Another fundamental  advantage is the fact that the properties of the general fractional derivative viz, $\Psi$-Hilfer is the preservation of the properties of the respective cases, in particular, in the investigated property of a fractional differential equation, in this case, the existence and uniqueness of solutions \cite{sousa2}--\cite{sousa5}.

In the present paper, we consider the following impulsive $\Psi$-Hilfer fractional differential equation ( impulsive $\Psi$-HFDE ) with initial condition
\begin{align}
& ^H \mathbf{D}^{\mu,\, \nu; \, \Psi}_{a^+}u(t)=f(t, u(t)),~t \in \mathcal{\mathcal{J}}=[a,T]-\{t_1, t_2,\cdots ,t_m\},\label{e11}\\
&\Delta \mathbf{I}_{a^+}^{1-\varrho; \, \Psi}u(t_k)= \zeta_k \in \mathbb{R},  ~k = 1,2,\cdots,m, \label{e12}\\
& \mathbf{I}_{a^+}^{1-\varrho; \, \Psi}u(a)=\delta \in \mathbb{R}, \label{e13}
\end{align}
where $0<\mu<1,~0\leq\nu\leq 1, ~\varrho=\mu+\nu-\mu\nu$,  ~$^H \mathbf{D}^{\mu, \, \nu; \, \Psi}_{a^+}(\cdot)$ is the $\Psi$-Hilfer fractional derivative of order $\mu $ and type $\nu$, ~ $\mathbf{I}_{a^+}^{1-\varrho; \, \Psi}$ is left sided $\Psi$-RL fractional integral operator, $ ~a=t_0< t_1 < t_2 < \cdots < t_m < t_{m+1}=T$,
~$
\Delta \mathbf{I}_{a^+}^{1-\varrho; \, \Psi}u(t_k)= \mathbf{I}_{a^+}^{1-\varrho; \, \Psi}u(t_k^+)- \mathbf{I}_{a^+}^{1-\varrho; \, \Psi}u(t_k^-) 
$, 
~$
 \mathbf{I}_{a^+}^{1-\varrho; \, \Psi}u(t_k^+) = \lim_{\epsilon\to 0^+} \mathbf{I}_{a^+}^{1-\varrho; \, \Psi}u(t_k + \epsilon)$ 
~~\mbox{and}~~
$\mathbf{I}_{a^+}^{1-\varrho; \, \Psi}u(t_k^-) = \lim_{\epsilon\to 0^-} \mathbf{I}_{a^+}^{1-\varrho; \, \Psi}u(t_k + \epsilon).
$

The main motivation for this work comes from the work highlighted above, with the purpose of investigating the existence and uniqueness of solution of impulsive $\Psi$-HFDEs and to provide new and more general results in the field of fractional differential equations.

We highlight here a rigorous analysis of Eq.(\ref{e11})-Eq.(\ref{e13}) regarding the main results and advantages obtained in this paper:
\begin{itemize}[topsep=0pt,itemsep=-1ex,partopsep=1ex,parsep=1ex]
\item With $\Psi(t)=t$ and taking the limits $\beta \rightarrow 0$ and $\beta \rightarrow 1$ of the Eq.(\ref{e11})-Eq.(\ref{e13}), we obtain the respective special cases for the differential equations, that is, the classical fractional derivatives of Riemann-Liouville and Caputo, respectively. In addition to the integer case, by choosing $\alpha$= 1. These are two special cases of fractional derivatives. However, a wide class of fractional derivatives can be obtained from the choice of the parameters $\beta $ and $ \Psi(t)$;

\item Since it is possible to obtain a wide class of derivatives from the choice of $\beta$ and $\Psi(t)$; consequently, it is also possible to obtain a class of fractional differential equations with their respective fractional derivatives, as particular cases;

\item A new class of solutions for impulsive $\Psi$-HFDEs;

\item We investigate the existence and uniqueness results for the impulsive $\Psi$-HFDEs and extend it to the non-local impulsive  $\Psi$-HFDEs.
\end{itemize}
Organization of Paper: In section 2,  some definitions and results that are important for the development of the paper have been provided via Lemmas and Theorems. In section 3, we present a representation formula for the solution, i.e., we show that the  problem (\ref{e11})-(\ref{e13}) is equivalent to the Volterra fractional integral equation. In section 4, we investigated the existence and uniqueness of the  impulsive $\Psi$-HFDE. In Section 5, we will investigate the existence and uniqueness of a nonlocal impulse $\Psi$-HFDE. Concluding and remarks closing the paper.
\section{Preliminaries} \label{preliminaries}
In this section, we introduce preliminary facts 
that are utilized all through this paper.

Let $\mathcal{I}=[a,b ]$ and $\Psi\in C^{1}(\mathcal{I},\mathbb{R})$ an increasing function such that $\Psi'(x)\neq 0$, $\forall~ x\in \mathcal{I}$.
\begin{definition} [\cite{Kilbas}]
 The $\Psi$-Riemann fractional integral of order $\mu>0$ of  the function $h$  is given by  
\begin{equation}\label{21}
\mathbf{I}_{a+}^{\mu ;\, \Psi }h\left( t\right) :=\frac{1}{\Gamma \left( \mu
\right) }\int_{a}^{t}\mathcal{L}_{\Psi}^{\mu}(t,\sigma )h\left( \sigma \right) d\sigma,
\end{equation}
where
\begin{equation*}
\mathcal{L}_{\Psi}^{\mu}(t,\sigma )=\Psi ^{\prime }\left(\sigma \right) \left( \Psi \left(
t\right) -\Psi \left( \sigma \right) \right) ^{\mu-1}
\end{equation*}
\end{definition}

\begin{lemma}[\cite{Kilbas}]
Let $\mu>0$, $\nu>0$ and $\delta >0$. Then: 
\begin{itemize}[topsep=0pt,itemsep=-1ex,partopsep=1ex,parsep=1ex]
\item[(i)] $\mathbf{I}_{a^+}^{\mu;\, \Psi}\mathbf{I}_{a^+}^{\nu ;\, \Psi} h(t)=\mathbf{I}_{a^+}^{\mu+\nu;\, \Psi} h(t)$
\item[(ii)] If $h(t)= (\Psi(t)-\Psi(a))^{\delta-1},$
 then   ~ $\mathbf{I}_{a^+}^{\mu;\, \Psi}h(t)=\frac{\Gamma(\delta)}{\Gamma(\mu+\delta)}(\Psi(t)-\Psi(a))^{\mu + \delta-1}.$
\end{itemize}

\end{lemma}
\begin{definition} [\cite{{Sousa1}}]
The $\Psi$--Hilfer fractional derivative of  function $h$ of order $\mu$, $(0<\mu<1)$ and of type $0\leq \nu \leq 1$, is defined by
$$^H \mathbf{D}^{\mu, \, \nu; \, \Psi}_{a^+}h(t)= \mathbf{I}_{a^+}^{\nu ({1-\mu});\, \Psi} \left(\frac{1}{{\Psi}^{'}(t)}\frac{d}{dt}\right)^{'}\mathbf{I}_{a^+}^{(1-\nu)(1-\mu);\, \Psi} h(t).$$
\end{definition}

\begin{theorem}[\cite{Sousa1}]\label{ab}
If $h\in C^{1}(\mathcal{I}),$ $0<\mu<1$ and $0\leq\nu \leq 1 $, then
\begin{enumerate}[topsep=0pt,itemsep=-1ex,partopsep=1ex,parsep=1ex]
\item[(i)] $\mathbf{I}_{a^+}^{\mu;\, \Psi}\, {^H \mathbf{D}^{\mu, \, \nu; \, \Psi}_{a^+}}h(t)= h(t)- \Omega_{\Psi}^{\varrho}(t,a)\mathbf{I}_{a^+}^{(1-\nu)(1-\mu);\Psi}h(a),$
where ~$\Omega_{\Psi}^{\varrho}(t,a)=\frac{(\Psi(t)-\Psi(a))^{\varrho-1}}{\Gamma(\varrho)}
$

\item[(ii)] ${^H\mathbf{D}^{\mu, \, \nu; \, \Psi}_{a^+}}\,\mathbf{I}_{a^+}^{\mu;\, \Psi}h(t)=h(t).$
\end{enumerate}
\end{theorem}
Consider the weighted space \cite{Sousa1} defined by
 $$
C_{1-\varrho;\Psi}(\mathcal{I})=\left\{u:(a,b]\to\mathbb{R} : ~(\Psi(t)-\Psi(a))^{1-\varrho}u(t)\in C(\mathcal{I})\right\},0< \varrho\leq 1.
$$
Define the weighted space of piecewise continuous functions as
\begin{align*}
\mathcal{PC}_{1-\varrho; \,  \Psi}(\mathcal{I},\mathbb{R}) =\{& u:(a,b]\to\mathbb{R} :u\in C_{1-\varrho;\Psi}((t_k,t_{k+1}],\mathbb{R}),k=0,1,2,\cdots, m,\\
& \mathbf{I}_{a^+}^{1-\varrho; \, \Psi} \, u(t_k^+), ~ \mathbf{I}_{a^+}^{1-\varrho; \, \Psi}\, u(t_k^-) ~\mbox{exists and}  ~\mathbf{I}_{a^+}^{1-\varrho; \, \Psi}\, u(t_k^-)= \mathbf{I}_{a^+}^{1-\varrho; \, \Psi} \, u(t_k) \\
& ~\mbox{for} ~ \, k = 1,2,\cdots, m \}
\end{align*}
Clearly, $\mathcal{PC}_{1-\varrho; \, \Psi}(\mathcal{I},\mathbb{R})$ is a Banach space with the norm
$$
\|u\|_{\mathcal{PC}_{1-\varrho; \,  \Psi}(\mathcal{I},\mathbb{R})} = \sup_ {t\in \mathcal{I}} \left|(\Psi(t)-\Psi(a))^{1-\varrho}u(t)\right|.
$$
Note that for $\varrho =1$, we get 
$ \mathcal{PC}_{0; \, \Psi}(\mathcal{I},\mathbb{R})=PC(\mathcal{I},\mathbb{R})$ 
a particular case of the space $\mathcal{PC}_{1-\varrho; \,  \Psi}(\mathcal{I},\mathbb{R})$, whose  details are given in \cite{Benchohra,Wang, Bai}.

With suitable modification, the PC-type Arzela--Ascoli Theorem \cite{bainov,wei} can be extended to the weighted space $\mathcal{PC}_{1-\varrho; \,  \Psi}\left( I,\, \mathcal{X}\right)$, where $I$ is closed bounded interval.

\begin{theorem}[$\mathcal{PC}_{1-\varrho; \, \Psi}$ type Arzela-–Ascoli Theorem] \label{pc}

Let $\mathcal{X}$ be a Banach space and $\mathcal{W}_{1-\varrho;\, \Psi} \subset \mathcal{PC}_{1-\varrho; \, \Psi}(\mathcal{J},\mathcal{X}).$ If the following conditions are satisfied:
\begin{itemize} [topsep=0pt,itemsep=-1ex,partopsep=1ex,parsep=1ex]
\item [{\rm(a)}]$\mathcal{W}_{1-\varrho;\, \Psi}$ is uniformly bounded subset of $\mathcal{PC}_{1-\varrho; \, \Psi}(\mathcal{J}, \mathcal{X})$;
\item [{\rm(b)}] $\mathcal{W}_{1-\varrho;\, \Psi}$ is equicontinuous in $(t_k, t_{k+1}), k = 0, 1, 2,\cdots , m, where~ t_0 = a, t_{m+1} = T$ ;
\item [{\rm(c)}]$\mathcal{W}_{1-\varrho;\, \Psi}(t) = \{u(t): u \in \mathcal{W}_{1-\varrho;\, \Psi}, ~t \in \mathcal{J} - {t_1, \cdots , t_m}\}, \mathcal{W}_{1-\varrho;\, \Psi}(t_k^+ ) = \{u(t_k^+ ): u \in  \mathcal{W}_{1-\varrho;\, \Psi}\}~ \text{and}~ \mathcal{W}_{1-\varrho;\, \Psi}(t_k^- ) = \{u(t_k^- ): u \in  \mathcal{W}_{1-\varrho;\, \Psi}\}$ are relatively compact subsets of X, 
\end{itemize}
then $\mathcal{W}_{1-\varrho;\, \Psi}$ is a relatively compact subset of $\mathcal{PC}_{1-\varrho; \, \Psi}(\mathcal{J}, X)$.
\end{theorem}
\begin{proof} Let $\mathcal{W}_{1-\varrho; \, \Psi} \subset \mathcal{PC}_{1-\varrho; \,  \Psi}\left( \mathcal{J},\,\mathcal{X}\right) $ satisfy the conditions (a) to (c). Let $\{z_n\}$ be any sequence in $\mathcal{W}_{1-\varrho; \, \Psi}$. Define $x_n(t)= (\Psi(t)-\Psi(a))^{1-\varrho}z_n(t), \forall\,  n$. Then sequence $\{x_n\}\subset \mathcal{W}\subset PC(\mathcal{J},\mathcal{X})$, where $\mathcal{W}$ satisfy the conditions of  Theorem 2.1  of \cite{wei}. Proceeding as in the proof of Theorem 2.1  of \cite{wei},  there exist $x \in PC(\mathcal{J},\mathcal{X})$ such that $x_n \to x$ in $ PC(\mathcal{J},\mathcal{X})$ which in turn gives $z_n \to z$ in $ \mathcal{PC}_{1-\varrho; \,  \Psi}\left( \mathcal{J},\,\mathcal{X}\right)$. This proves $\mathcal{W}_{1-\varrho;\, \Psi}$ is a relatively compact subset of $\mathcal{PC}_{1-\varrho; \, \Psi}(\mathcal{J}, \mathcal{X})$.
\end{proof}

\begin{theorem}[Krasnoselskii, \cite{zhou}]\label{kf}
Let $\mathcal{M}$ be a closed, convex, and nonempty subset of a Banach space $\mathcal{X}$,
and A, B the operators such that
\begin{itemize} [topsep=0pt,itemsep=-1ex,partopsep=1ex,parsep=1ex]
\item[{\rm 1.}]$ \mathcal{A}x + \mathcal{B}y \in \mathcal{M}$ whenever $x, y \in \mathcal{M}$;
\item[{\rm 2.}] $\mathcal{A}$ is compact and continuous;
\item[{\rm 3.}] $\mathcal{B}$ is a contraction mapping.
\end{itemize}
Then there exists $ z \in \mathcal{M} $ such that  $z = \mathcal{A}z + \mathcal{B}z$.
\end{theorem}

\section{Representation formula for the solution}

The following lemma play an important role in building an equivalent fractional integral equation of the impulsive $\Psi$-HFDE \eqref{e11} - \eqref{e13}.
\begin{lemma}
Let $0<\mu<1$ and  $0\leq\nu\leq 1,$ $\varrho=\mu+\nu-\mu\nu$ and $ h: \mathcal{J}\to \mathbb{R} $ be continuous. 

Then for any $b\in \mathcal{J} $ a function $u\in C_{1-\varrho,\Psi}\left( \mathcal{J},\,\mathbb{R}\right) $ defined  by 
\begin{equation}\label{JK}
u(t)= \Omega_{\Psi}^{\varrho}(t,a)\left. \left\{\mathbf{I}_{a^+}^{{1-\varrho}; \, {\Psi}}u(b)-\mathbf{I}_{a^+}^{1-\varrho+\mu; \, \Psi}h(t)\right|_{t=b} \right\}+\mathbf{I}_{a^+}^{\mu; \, \Psi}h(t)
\end{equation}
is the solution of the $\Psi$--Hilfer fractional differential equation $$^H \mathbf{D}^{\mu, \, \nu; \, \Psi}_{a^+}u(t)=h(t),~t \in \mathcal{J}.$$
\end{lemma}

\begin{proof}
Applying $^H \mathbf{D}^{\mu, \, \nu; \, \Psi}_{a^+} $ on both sides of the equation \eqref{JK}, we get
\begin{align*}
^H \mathbf{D}^{\mu, \, \nu; \, \Psi}_{a^+}u(t)
&= \left. \left\{\mathbf{I}_{a^+}^{{1-\varrho}; \, {\Psi}}u(b)-\mathbf{I}_{a^+}^{1-\varrho+\mu; \, \Psi}h(t)\right|_{t=b} \right\} \,  ^H \mathbf{D}^{\mu, \, \nu; \, \Psi}_{a^+}\Omega_{\Psi}^{\varrho}(t,a)\\
&\qquad + ^H \mathbf{D}^{\mu, \, \nu; \, \Psi}_{a^+} \mathbf{I}_{a^+}^{\mu; \, \Psi}h(t), ~ t\in \mathcal{J}.
\end{align*}

Using the result (\cite{Sousa2}, Page 10),  
\begin{equation} \label{e32}
^H \mathbf{D}^{\mu, \, \nu; \, \Psi}_{a^+}(\Psi(t)-\Psi(a))^{\varrho-1}=0, ~ 0< \varrho <1,
\end{equation}
and using the Theorem \ref{ab}, we get 
$$
^H \mathbf{D}^{\mu, \, \nu; \, \Psi}_{a^+}u(t)=h(t),~t \in \mathcal{J}.$$

This completes the proof of the Lemma.
\end{proof}

In the next result, utilizing the Lemma \ref{JK}, we obtain the equivalent fractional integral of the problem \eqref{e11}-\eqref{e13}.

\begin{lemma}\label{fie}
Let $h: \mathcal{J}\to \mathbb{R}$ be a continuous function.  Then a function $u \in {\mathcal{PC}}_{1-\varrho;\, \Psi}\left( \mathcal{J},\,\mathbb{R}\right)$ is a solution of impulsive $\Psi$--HFDE 
\begin{align}
& ^H \mathbf{D}^{\mu,\, \nu; \, \Psi}_{a^+}u(t)=h(t),~t \in \mathcal{J}-\{t_1, t_2,\cdots ,t_m\},\label{e301}\\
&\Delta \mathbf{I}_{a^+}^{1-\varrho; \, \Psi}u(t_k)= \zeta_k \in \mathbb{R},  ~  ~k = 1,2,3,\cdots,m, \label{e302}\\
& \mathbf{I}_{a^+}^{1-\varrho; \, \Psi}u(a)=\delta \in \mathbb{R}, \label{e303}
\end{align}
if and only if u is a solution of the following fractional integral equation
\begin{equation}\label{e14}
u(t) =
\begin{cases}
\Omega_{\Psi}^{\varrho}(t,a)\, \delta+\mathbf{I}_{a^+}^{\mu; \, \Psi}h(t),~ \text{$t \in [a,t_1], $}\\
 \Omega_{\Psi}^{\varrho}(t,a)\, \left(\delta+\sum_{i=1}^{k}\zeta_i\right)+\mathbf{I}_{a^+}^{\mu; \, \Psi}h(t), ~\text{ $t \in (t_k,t_{k+1}],~ k=1,2,\cdots,m $}.  
\end{cases}
\end{equation}
\end{lemma}
\begin{proof}

Assume that $u \in {\mathcal{PC}}_{1-\varrho;\, \Psi}\left( \mathcal{J},\,\mathbb{R}\right)$ satisfies the  impulsive $\Psi$--HFDE \eqref{e301}-\eqref{e303}.

If $t\in [a,t_1]$ then 
\begin{equation} \label{e3.3}
\begin{cases}
^H \mathbf{D}^{\mu, \, \nu; \, \Psi}_{a^+}u(t)=h(t)\\
\mathbf{I}_{a^+}^{1-\varrho; \, \Psi}u(a)=\delta.  
\end{cases} 
\end{equation} 

Then the problem \eqref{e3.3} is equivalent to the following fractional integral \cite{Sousa2} 
 \begin{equation}\label{e15}
u(t)= \Omega_{\Psi}^{\varrho}(t,a) \, \delta+\mathbf{I}_{a^+}^{\mu; \, \Psi}h(t) ,\, \text{ $t \in [a,t_1] $.}
 \end{equation}

Now, if $t\in (t_1,t_2]$ then
$$
^H \mathbf{D}^{\mu, \, \nu; \, \Psi}_{a^+}u(t)=h(t), \, t\in (t_1,t_2] 
~~\mbox{with}~~\mathbf{I}_{a^+}^{1-\varrho; \, \Psi}u(t_1^+)- \mathbf{I}_{a^+}^{1-\varrho; \, \Psi}u(t_1^-)=\zeta_1.
$$ 

By Lemma \ref{JK}, we have 
\begin{align}\label{e16}
u(t)\nonumber &= \Omega_{\Psi}^{\varrho}(t,a)\left. \left\{\mathbf{I}_{a^+}^{{1-\varrho}; \, {\Psi}}u(t_1^+)-\mathbf{I}_{a^+}^{1-\varrho+\mu; \, \Psi}h(t)\right|_{t=t_1} \right\}+\mathbf{I}_{a^+}^{\mu; \, \Psi}h(t) \nonumber \\ 
&=  \Omega_{\Psi}^{\varrho}(t,a) \left. \left\{\mathbf{I}_{a^+}^{{1-\varrho}; \, {\Psi}}u(t_1^-)+\zeta_1-\mathbf{I}_{a^+}^{1-\varrho+\mu; \, \Psi}h(t)\right|_{t=t_1} \right\}\nonumber \\
& \qquad +\mathbf{I}_{a^+}^{\mu; \, \Psi}h(t), ~ t\in (t_1,t_2].
   \end{align}

Now, from \eqref{e15}, we have 
$$ \mathbf{I}_{a^+}^{{1-\varrho}; \, {\Psi}}u(t)= \delta+\mathbf{I}_{a^+}^{{1-\varrho+\mu}; \, {\Psi}}h(t).$$

This gives
 \begin{equation}\label{e17}
\left. \mathbf{I}_{a^+}^{{1-\varrho}; \, {\Psi}}u(t_1^-)-\mathbf{I}_{a^+}^{1-\varrho+\mu; \, \Psi}h(t)\right|_{t=t_1}=\delta.
\end{equation}

Using \eqref{e17} in \eqref{e16}, we obtain
 \begin{equation}\label{e18}
u(t)=\Omega_{\Psi}^{\varrho}(t,a) (\delta+\zeta_1) +\mathbf{I}_{a^+}^{\mu; \, \Psi}h(t), ~ t\in (t_1,t_2].
\end{equation} 

Next,  if $t\in (t_2,t_3]$ then 
$$
^H \mathbf{D}^{\mu, \, \nu; \, \Psi}_{a^+}u(t)=h(t), \, t\in (t_2,t_3]~~\mbox{with}~~
\mathbf{I}_{a^+}^{1-\varrho; \, \Psi}u(t_2^+)- \mathbf{I}_{a^+}^{1-\varrho; \, \Psi}u(t_2^-)=\zeta_2.
$$ 

Again by Lemma \ref{JK}, we have 
 \begin{align}\label{e19}
 u(t)&= \Omega_{\Psi}^{\varrho}(t,a) \, \left. \left\lbrace \mathbf{I}_{a^+}^{{1-\varrho}; \, {\Psi}}u(t_2^+)-\mathbf{I}_{a^+}^{1-\varrho+\mu; \, \Psi}h(t)\right|_{t=t_2} \right \}+\mathbf{I}_{a^+}^{\mu; \, \Psi}h(t) \nonumber \\ 
 &=  \Omega_{\Psi}^{\varrho}(t,a) \,\left. \left\lbrace \mathbf{I}_{a^+}^{{1-\varrho}; \, {\Psi}}u(t_2^-)+\zeta_2-\mathbf{I}_{a^+}^{1-\varrho+\mu; \, \Psi}h(t)\right|_{t=t_2}\right \} \nonumber \\ 
&\qquad+\mathbf{I}_{a^+}^{\mu; \, \Psi}h(t), \, t\in (t_2,t_3].
\end{align} 

From \eqref{e18}, we have 
$$ \mathbf{I}_{a^+}^{{1-\varrho}; \, {\Psi}}u(t)= (\delta+\zeta_1)+\mathbf{I}_{a^+}^{{1-\varrho+\mu}; \, {\Psi}}h(t),$$
which gives 
 \begin{equation}\label{e20}
\left. \mathbf{I}_{a^+}^{{1-\varrho}; \, {\Psi}}u(t_2^-)-\mathbf{I}_{a^+}^{1-\varrho+\mu; \, \Psi}h(t)\right|_{t=t_2}=\delta+\zeta_1.
 \end{equation}    

Using \eqref{e20} in \eqref{e19}, we get 
\begin{equation}\label{e21}
u(t)=\Omega_{\Psi}^{\varrho}(t,a) \, (\delta+\zeta_1+\zeta_2)+\mathbf{I}_{a^+}^{\mu; \, \Psi}h(t), \, t\in (t_2,t_3].
\end{equation}

Continuing the above process, we obtain 
$$ u(t)=\Omega_{\Psi}^{\varrho}(t,a) \, \left( \delta+\sum_{i=1}^{k}\zeta_i\right)+\mathbf{I}_{a^+}^{\mu; \, \Psi}h(t), ~ \text{ $ t \in (t_k,t_{k+1}], ~k=1,2,\cdots,m$.} 
$$

Conversely, let $u \in {\mathcal{PC}}_{1-\varrho;\, \Psi}\left( \mathcal{J},\,\mathbb{R}\right)$ satisfies the fractional integral equation \eqref{e14}. Then, for $t\in [a,t_1]$, we have
$$ u(t)= \Omega_{\Psi}^{\varrho}(t,a) \, \delta+\mathbf{I}_{a^+}^{\mu; \, \Psi}h(t). $$

Applying the $\Psi$-Hilfer fractional derivative operator $^H D^{\mu, \, \nu; \, \Psi}_{a^+}$ on both sides, we get
\begin{equation}
^H \mathbf{D}^{\mu, \, \nu; \, \Psi}_{a^+}u(t)=\delta \, {^H \mathbf{D}^{\mu, \, \nu; \, \Psi}_{a^+}}\Omega_{\Psi}^{\varrho}(t,a) + {^H \mathbf{D}^{\mu, \, \nu; \, \Psi}_{a^+}} \mathbf{I}_{a^+}^{\mu; \, \Psi}h(t).
\end{equation}

Utilizing \eqref{e32} and Theorem \ref{ab},
\begin{equation*}
^H \mathbf{D}^{\mu, \, \nu; \, \Psi}_{a^+}u(t)=h(t),~t \in [a,t_1].
\end{equation*}

Now, for $t\in (t_k,t_{k+1}],~ (k=1,2,\cdots,m)$, we have
$$u(t)=\Omega_{\Psi}^{\varrho}(t,a) \, \left( \delta+\sum_{i=1}^{k}\zeta_i\right)+\mathbf{I}_{a^+}^{\mu; \, \Psi}h(t),~ \text{ $ t \in (t_k,t_{k+1}], ~k=1,2,\cdots,m$}. $$

Applying the operator  $^H \mathbf{D}^{\mu, \, \nu; \, \Psi}_{a^+}(\cdot)$ on both sides and using \eqref{e32}  and the  Theorem \ref{ab}, we obtain
\begin{align*}
^H \mathbf{D}^{\mu, \, \nu; \, \Psi}_{a^+}u(t)&= \left\{\delta+\sum_{i=1}^{k}\zeta_i\right \}{^H \mathbf{D}^{\mu, \, \nu; \, \Psi}_{a^+}}\Omega_{\Psi}^{\varrho}(t,a) + {^H \mathbf{D}^{\mu, \, \nu; \, \Psi}_{a^+}} \mathbf{I}_{a^+}^{\mu; \, \Psi}h(t)\\
&=h(t).
\end{align*}

We have proved that $u$ satisfies \eqref{e301}. Next, we prove that $u$ also satisfy the  conditions \eqref{e302} and \eqref{e303}. 

Applying the $\Psi$-RL fractional operator
$\mathbf{I}_{a^+}^{{1-\varrho}; \, {\Psi}}(\cdot)$ on both sides of  \eqref{e15}, we get
\begin{align*}
\mathbf{I}_{a^+}^{{1-\varrho}; \, {\Psi}}u(t)&=\delta \mathbf{I}_{a^+}^{{1-\varrho}; \, {\Psi}}\Omega_{\Psi}^{\varrho}(t,a)+ \mathbf{I}_{a^+}^{{1-\varrho}; \, {\Psi}} \mathbf{I}_{a^+}^{\mu; \, \Psi}h(t)\\
&= \delta + \mathbf{I}_{a^+}^{{1-\varrho+\mu}; \, {\Psi}} h(t),
\end{align*}
and from which we obtain
\begin{equation*}\label{e2}
\mathbf{I}_{a^+}^{{1-\varrho}; \, {\Psi}}u(a)=\delta, 
\end{equation*}
which is the condition  \eqref{e303}.

Further, from equation \eqref{e14},  for $ t\in(t_k,t_{k+1}]$,  we have  
\begin{align}\label{e22} 
\mathbf{I}_{a^+}^{{1-\varrho}; \, {\Psi}}u(t)&=\left\{\delta+\sum_{i=1}^{k}\zeta_i\right \}\mathbf{I}_{a^+}^{{1-\varrho}; \, {\Psi}}\Omega_{\Psi}^{\varrho}(t,a)+ \mathbf{I}_{a^+}^{{1-\varrho}; \, {\Psi}} \mathbf{I}_{a^+}^{\mu; \, \Psi}h(t) \nonumber\\
&= \delta+\sum_{i=1}^{k}\zeta_i + \mathbf{I}_{a^+}^{{1-\varrho+\mu}; \, {\Psi}} h(t), 
\end{align}
and for  $ t\in(t_{k-1},t_k]$,  we have
\begin{align} \label{e23}
\mathbf{I}_{a^+}^{{1-\varrho}; \, {\Psi}}u(t)&= \left\{\delta+\sum_{i=1}^{k-1}\zeta_i\right \}\mathbf{I}_{a^+}^{{1-\varrho}; \, {\Psi}}\Omega_{\Psi}^{\varrho}(t,a)+ \mathbf{I}_{a^+}^{{1-\varrho}; \, {\Psi}} \mathbf{I}_{a^+}^{\mu; \, \Psi}h(t) \nonumber\\
&= \delta+\sum_{i=1}^{k-1}\zeta_i + \mathbf{I}_{a^+}^{{1-\varrho+\mu}; \, {\Psi}} h(t),
\end{align}

Therefore, from \eqref{e22} to \eqref{e23}, we obtain
\begin{equation}\label{e24}
\mathbf{I}_{a^+}^{1-\varrho; \, \Psi}u(t_k^+)- \mathbf{I}_{a^+}^{1-\varrho; \, \Psi}u(t_k^-)= \sum_{i=1}^{k}\zeta_i-\sum_{i=1}^{k-1}\zeta_i = \zeta_k
\end{equation}
which condition \eqref{e302}.  We have proved that $u$ satisfies the impulsive $\Psi$--HFDE \eqref{e301}-\eqref{e303}. This completes the proof.
\end{proof}

\section{Existence and Uniqueness results}

\begin{theorem}{\rm\textbf{(Existence)}}\label{ex}
Assume that the function $f:(a,T]\times \mathbb{R} \to \mathbb{R}$ is continuous and satisfies the conditions:
\begin{itemize} [topsep=0pt,itemsep=-1ex,partopsep=1ex,parsep=1ex]
\item[{\rm($A_1$)}] $f(\cdot,u(\cdot))\in {\mathcal{PC}}_{1-\varrho;\, \Psi}\left( \mathcal{J},\,\mathbb{R}\right)$ for any $u\in {\mathcal{PC}}_{1-\varrho;\, \Psi}\left( \mathcal{J},\,\mathbb{R}\right),$
\item[{\rm($A_2$)}] there exist a constant $0<L\leq \dfrac{\Gamma(\mu+\varrho)}{2\Gamma(\varrho)(\Psi(T)-\Psi(a))^\mu}$ satisfying 
$$
|f(t,u)-f(t,v)|\leq L |u-v|, ~ t\in \mathcal{J}, ~u,v \in \mathbb{R}.
$$ 
\end{itemize}
Then, the impulsive $\Psi$--HFDE \eqref{e11}-\eqref{e13} has at least one solution in ${\mathcal{PC}}_{1-\varrho;\, \Psi}\left( \mathcal{J},\,\mathbb{R}\right)$.
\end{theorem}
\begin{proof} In the view of Lemma \ref{fie}, the equivalent fractional integral equation of the impulsive $\Psi$-HFDE \eqref{e11} - \eqref{e13} is given by 
\begin{equation}\label{e41}
u(t) =
\Omega_{\Psi}^{\varrho}(t,a)\left(\delta+\sum_{a <t_k <t}\zeta_k \right)+\mathbf{I}_{a^+}^{\mu; \, \Psi}f(t, u(t)), \text{ $t \in \mathcal{J} $}.
\end{equation} 

Consider the set 
$$ \mathcal{B}_r = \left\{u\in \mathcal{PC}_{1-\varrho; \,  \Psi}\left( \mathcal{J},\,\mathbb{R}\right) : \mathbf{I}_{a^+}^{1-\varrho; \, \Psi}u(a)=\delta , ~\|u\|_{\mathcal{PC}_{1-\varrho; \,  \Psi}\left( \mathcal{J},\,\mathbb{R}\right)} \leq r \right\},
$$ 
where $$
\mathcal{M} = \sup_{\sigma\in \mathcal{J}} |f(\sigma,0)|
$$ 
and 
 $$
 r \geq 2\left(\frac{1}{\Gamma(\varrho)}\left\{|\delta|+\sum_{i=1}^{m}|\zeta_i|\right\} + \frac{\mathcal{M}(\Psi(T)-\Psi(a))^{1-\varrho+\mu}}{\Gamma(\mu+1)} \right).
 $$

We define the operators $\mathcal {P}$ and $\mathcal {Q}$ on $\mathcal{B}_r$ by
\begin{align*}
&\mathcal {P}u(t)= \Omega_{\Psi}^{\varrho}(t,a)\left(\delta+\sum_{a <t_k <t}\zeta_k \right), t\in \mathcal{J},\\
&\mathcal {Q}u(t)= \mathbf{I}_{a^+}^{\mu; \, \Psi}f(t,u(t)) ,\, t \in \mathcal{J}. 
\end{align*}  

Then the fractional integral equation \eqref{e41} can be written as operator equation $$u = \mathcal {P}u+\mathcal {Q}u , \, \, \,  u \in {\mathcal{PC}}_{1-\varrho;\, \Psi}\left( \mathcal{J},\,\mathbb{R}\right).$$
$\mathit{\rm \textbf{Step ~1:}}$ We prove that  $ \mathcal {P}u+\mathcal {Q}v \in \mathcal{B}_r $ for any $u,v \in \mathcal{B}_r.$ \\

Let any $u,v \in \mathcal{B}_r$. Then using ($A_1$), for any $t\in \mathcal{J}$, we have  
\begin{align*} 
&\left|(\Psi(t)-\Psi(a))^{1-\varrho}(\mathcal {P}u(t)+\mathcal {Q}v(t))\right|\\
& =  \left|(\Psi(t)-\Psi(a))^{1-\varrho}\left\{\Omega_{\Psi}^{\varrho}(t,a)\left(\delta+\sum_{a <t_k <t}\zeta_k\right)+\mathbf{I}_{a^+}^{\mu; \, \Psi}f(t,v(t))\right \}\right|\\
&\leq \frac{1}{\Gamma(\varrho)}\left(|\delta|+\sum_{k=1}^{m}|\zeta_k
|\right) + \frac{(\Psi(t)-\Psi(a))^{1-\varrho}}{\Gamma(\mu)} \int_{a}^{t}\mathcal{L}_{\Psi}^{\mu}(t,\sigma)\,\left|f(\sigma,v(\sigma))\right|d\sigma\\
&\leq \frac{1}{\Gamma(\varrho)}\left(|\delta|+\sum_{k=1}^{m}|\zeta_k|\right)+ \frac{(\Psi(t)-\Psi(a))^{1-\varrho}}{\Gamma(\mu)} \int_{a}^{t}\mathcal{L}_{\Psi}^{\mu}(t,\sigma)\, |f(\sigma,v(\sigma))- f(\sigma,0)|d\sigma\\
& \qquad + \frac{(\Psi(t)-\Psi(a))^{1-\varrho}}{\Gamma(\mu)} \int_{a}^{t}\mathcal{L}_{\Psi}^{\mu}(t,\sigma)\,|f(\sigma,0)|d\sigma\\
&\leq \frac{1}{\Gamma(\varrho)}\left(|\delta|+\sum_{k=1}^{m}|\zeta_k|\right)+ \frac{L \, (\Psi(t)-\Psi(a))^{1-\varrho}}{\Gamma(\mu)} \int_{a}^{t}\mathcal{L}_{\Psi}^{\mu}(t,\sigma)|v(\sigma)|d\sigma\\
& \qquad + \frac{\mathcal{M} \, (\Psi(t)-\Psi(a))^{1-\varrho}}{\Gamma(\mu)} \int_{a}^{t}\mathcal{L}_{\Psi}^{\mu}(t,\sigma)\,d\sigma\\
&=\frac{1}{\Gamma(\varrho)}\left(|\delta|+\sum_{k=1}^{m}|\zeta_k|\right)\\
&\qquad+ \frac{L \, (\Psi(t)-\Psi(a))^{1-\varrho}}{\Gamma(\mu)} \int_{a}^{t}\mathcal{L}_{\Psi}^{\mu}(t,\sigma)\,(\Psi(\sigma)-\Psi(a))^{\varrho-1} 
\left|(\Psi(\sigma)-\Psi(a))^{1-\varrho}v(\sigma)\right|d\sigma \\
&\qquad + \frac{\mathcal{M}\, (\Psi(t)-\Psi(a))^{1-\varrho}}{\Gamma(\mu)} \frac{(\Psi(t)-\Psi(a))^{\mu}}{\mu}\\
& \leq \frac{1}{\Gamma(\varrho)}\left(|\delta|+\sum_{k=1}^{m}|\zeta_k|\right)+
L \,(\Psi(t)-\Psi(a))^{1-\varrho} \|v\|_{{\mathcal{PC}}_{1-\varrho;\, \Psi}\left( \mathcal{J},\,\mathbb{R}\right)} ~ \mathbf{I}_{a^+}^{\mu; \, \Psi}(\Psi(t)-\Psi(a))^{\varrho-1}\\
& \qquad + \frac{\mathcal{M}\, (\Psi(t)-\Psi(a))^{1-\varrho+\mu}}{\Gamma(\mu+1)}\\ 
& \leq \frac{1}{\Gamma(\varrho)}\left(|\delta|+\sum_{k=1}^{m}|\zeta_k|\right)+
\frac{L\,\Gamma(\varrho)}{\Gamma(\mu+\varrho)}(\Psi(t)-\Psi(a))^{\mu} \|v\|_{{\mathcal{PC}}_{1-\varrho;\, \Psi}\left( \mathcal{J},\,\mathbb{R}\right)} + \frac{\mathcal{M}\,(\Psi(t)-\Psi(a))^{1-\varrho+\mu}}{\Gamma(\mu+1)}\\
&\leq \frac{1}{\Gamma(\varrho)}\left(|\delta|+\sum_{k=1}^{m}|\zeta_k|\right)+
\frac{L\,\Gamma(\varrho)}{\Gamma(\mu+\varrho)}(\Psi(T)-\Psi(a))^{\mu} r + \frac{\mathcal{M}\,(\Psi(T)-\Psi(a))^{1-\varrho+\mu}}{\Gamma(\mu+1)}.
\end{align*}

Since $$r \geq 2\left(\frac{1}{\Gamma(\varrho)}\left\{|\delta|+\sum_{i=1}^{m}|\zeta_i|\right\} + \frac{\mathcal{M}(\Psi(T)-\Psi(a))^{1-\varrho+\mu}}{\Gamma(\mu+1)} \right)$$ 
and $$L\leq \frac{\Gamma(\mu+\varrho)}{2\Gamma(\varrho)(\Psi(T)-\Psi(a))^\mu},$$ we have
\begin{align*}
\left|(\Psi(t)-\Psi(a))^{1-\varrho}(\mathcal {P}u(t)+\mathcal {Q}v(t))\right| \leq r, ~ t\in \mathcal{J}.
\end{align*}

Therefore
\begin{align*}
 \left\|(\mathcal {P}u+\mathcal {Q}v)\right\|_{{\mathcal{PC}}_{1-\varrho;\, \Psi}\left( \mathcal{J},\,\mathbb{R}\right)} \leq r.
\end{align*}

Further, from definition of the operator $\mathcal{P}$ and  $\mathcal{Q}$, one can verify that
\begin{align*}
 \mathbf{I}_{a^+}^{1-\varrho; \, \Psi}(\mathcal {P}u+\mathcal {Q}v)(a)=\delta.
\end{align*}
We have proved that, $\mathcal {P}u+\mathcal {Q}v \in \mathcal{B}_r.$\\

\noindent $\mathit{\rm \textbf{Step ~2 :~}}$ Clearly $\mathcal {P}$ is a contraction with the contraction constant zero.\\
$\mathit{\rm\textbf{Step ~3 :~}}$ $\mathcal {Q}$ is compact and continuous. \\

The continuity of $\mathcal {Q}$ follows from the continuity of $f$.
Next we prove that $\mathcal {Q}$ is uniformly bounded on $\mathcal{B}_r$. 

Let any $u\in \mathcal{B}_r$. Then by ($A_2$), for any $t \in \mathcal{J}$, we have 
 \begin{align*}
 \left|(\Psi(t)-\Psi(a))^{1-\varrho}\mathcal {Q}u(t)\right|
 & = \left|(\Psi(t)-\Psi(a))^{1-\varrho}\mathbf{I}_{a^+}^{\mu; \, \Psi}f(t,u(t))\right| \\
 &\leq \frac{(\Psi(t)-\Psi(a))^{1-\varrho}}{\Gamma(\mu)} \int_{a}^{t}\mathcal{L}_{\Psi}^{\mu}(t,\sigma)\,|f(\sigma,u(\sigma))|\,d\sigma\\
  &\leq \frac{(\Psi(t)-\Psi(a))^{1-\varrho}}{\Gamma(\mu)} \int_{a}^{t}\mathcal{L}_{\Psi}^{\mu}(t,\sigma)\,|f(\sigma,u(\sigma))- f(\sigma,0)|\,d\sigma\\
  & \qquad + \frac{(\Psi(t)-\Psi(a))^{1-\varrho}}{\Gamma(\mu)} \int_{a}^{t}\mathcal{L}_{\Psi}^{\mu}(t,\sigma)\,|f(\sigma,0)|\,d\sigma\\  
  &\leq \frac{L\,(\Psi(t)-\Psi(a))^{1-\varrho}}{\Gamma(\mu)} \int_{a}^{t}\mathcal{L}_{\Psi}^{\mu}(t,\sigma)\,|u(\sigma)|\,d\sigma\\
    & \qquad + \frac{M\,(\Psi(t)-\Psi(a))^{1-\varrho}}{\Gamma(\mu)} \int_{a}^{t}\mathcal{L}_{\Psi}^{\mu}(t,\sigma)\,d\sigma\\  
  &\leq \frac{L\,\Gamma(\varrho)}{\Gamma(\mu+\varrho)}(\Psi(t)-\Psi(a))^{\mu}\, \|u\|_{{\mathcal{PC}}_{1-\varrho;\, \Psi}\left( \mathcal{J},\,\mathbb{R}\right)}+\frac{\mathcal{M}\,(\Psi(t)-\Psi(a))^{1-\varrho+\mu}}{\Gamma(\mu+1)}\\
  & \leq  \frac{L\, \Gamma(\varrho)}{\Gamma(\mu+\varrho)}(\Psi(T)-\Psi(a))^{\mu}\, r+\frac{\mathcal{M}\, (\Psi(T)-\Psi(a))^{1-\varrho+\mu}}{\Gamma(\mu+1)}.
 \end{align*}

Therefore
 \begin{equation*}
 \left\|\mathcal {Q}u\right\|_{\mathcal{PC}_{1-\varrho; \, \Psi}\left( \mathcal{J},\,\mathbb{R}\right)}
 \leq \frac{L\, \Gamma(\varrho)}{\Gamma(\mu+\varrho)}(\Psi(T)-\Psi(a))^{\mu}\, r+\frac{\mathcal{M}(\Psi(T)-\Psi(a))^{1-\varrho+\mu}}{\Gamma(\mu+1)}.
 \end{equation*}

This proves $\mathcal {Q}$ is uniformly bounded on $\mathcal{B}_r.$ Next, we show that $\mathcal {Q}\mathcal{B}_r$ is equicontinuous.

Let any $ u \in \mathcal{B}_r$ and $t_1, t_2 \in(t_k, t_{k+1}]$  for some $k, (k=0,1,\cdots,m)$ with $t_1 < t_2$. Then,
 \begin{align*}
 &\left|\mathcal {Q}u(t_2)-\mathcal {Q}u(t_1)\right|\\
 &= \left|\left( \left.\mathbf{I}_{a^+}^{\mu; \, \Psi}f(t,u(t))\right|_{t=t_2}\right) -\left( \left.\mathbf{I}_{a^+}^{\mu; \, \Psi}f(t,u(t))\right|_{t=t_1}\right)  \right|\\
& \leq \frac{1}{\Gamma(\mu)} \int_{a}^{t_2}\mathcal{L}_{\Psi}^{\mu}(t_2 ,\sigma)\,|f(\sigma,u(\sigma))|\,d\sigma\\
 & - \frac{1}{\Gamma(\mu)} \int_{a}^{t_1}\mathcal{L}_{\Psi}^{\mu}(t_1 ,\sigma)\,|f(\sigma,u(\sigma))|\,d\sigma\\
 & = \frac{1}{\Gamma(\mu)} \int_{a}^{t_2}\mathcal{L}_{\Psi}^{\mu}(t_2 ,\sigma)\,(\Psi(\sigma)-\Psi(a))^{\varrho-1}\left|(\Psi(\sigma)-\Psi(a))^{1-\varrho}f(\sigma,u(\sigma))\right|\, d\sigma\\
 &- \frac{1}{\Gamma(\mu)} \int_{a}^{t_1}\mathcal{L}_{\Psi}^{\mu}(t_1 ,\sigma)\,(\Psi(\sigma)-\Psi(a))^{\varrho-1}\left|(\Psi(\sigma)-\Psi(a))^{1-\varrho}f(\sigma,u(\sigma))\right|\, d\sigma\\
 &\leq \left\{ \left.\mathbf{I}_{a^+}^{\mu; \, \Psi}(\Psi(t)-\Psi(a))^{\varrho-1}\right|_{t=t_2}-\left.\mathbf{I}_{a^+}^{\mu; \, \Psi}(\Psi(t)-\Psi(a))^{\varrho-1}\right|_{t=t_1} \right \}\times \left\|f\right\|_{{\mathcal{PC}}_{1-\varrho;\, \Psi}\left( \mathcal{J},\,\mathbb{R}\right)}\\
 &= \frac{\Gamma(\varrho)}{\Gamma(\mu+\varrho)}\left\{(\Psi(t_2)-\Psi(a))^{1-\varrho+\mu}-(\Psi(t_1)-\Psi(a))^{1-\varrho+\mu}\right \}\left\|f\right\|_{{\mathcal{PC}}_{1-\varrho;\, \Psi}\left( \mathcal{J},\,\mathbb{R}\right)}.
 \end{align*}
Note that 
$$
\left|\mathcal {Q}u(t_2)-\mathcal {Q}u(t_1)\right| \to 0 \quad \mbox{as} \quad  |t_1-t_2| \to 0. 
$$ 
This shows that $\mathcal {Q}$ is equicontinuous on  $(t_k, t_{k+1}]$. Therefore $\mathcal {Q}$ is relatively compact on $\mathcal{B}_r$. By ${\mathcal{PC}}_{1-\varrho;\, \Psi}$ type Arzela-Ascoli Theorem (Theorem\ref{pc}) $\mathcal {Q}$ is compact on $\mathcal{B}_r$.  Since all the assumptions of Krasnoselskii's fixed point theorem
(Theorem \ref{kf})  are satisfied, the operator equation
$$u = \mathcal {P}u+\mathcal {Q}u $$ has fixed point $\tilde{u} \in {\mathcal{PC}}_{1-\varrho;\, \Psi}\left( \mathcal{J},\,\mathbb{R}\right)$, which is the solution of the impulsive $\Psi$-HFDE \eqref{e11}-\eqref{e13}.
\end{proof}

\begin{theorem}{\rm\textbf{(Uniqueness)}}  \label{unique1}
Assume that the function $f:(a,T]\times \mathbb{R} \to \mathbb{R}$ is continuous and satisfies the conditions $(A_1)-(A_2)$. Then, impulsive $\Psi$--HFDE  \eqref{e11}-\eqref{e13} has a unique solution in the weighted space ${\mathcal{PC}}_{1-\varrho;\, \Psi}\left( \mathcal{J},\,\mathbb{R}\right)$.
\end{theorem}

\begin{proof}
Consider the set $\mathcal{B}_r$ as defined in the Theorem \ref{ex} and define the operator $\mathcal{T}$ on  $\mathcal{B}_r$ by 
\begin{equation*}
\mathcal{T} u(t) =
\Omega_{\Psi}^{\varrho}(t,a) \left(\delta+\sum_{a <t_k <t}\zeta_k \right)+\mathbf{I}_{a^+}^{\mu; \, \Psi}f(t, u(t)), \text{ $t \in \mathcal{J} $}.  
\end{equation*}

To prove $u=\mathcal{T} u$ has a fixed point,
we show that $\mathcal{T}\mathcal{B}_r\subset \mathcal{B}_r$. For that take any $u\in \mathcal{B}_r$. Then, by ($A_2$) for any $t\in \mathcal{J}$, we have
\begin{align*}
&|\mathcal{T} u(t)|\\
&= \left|\Omega_{\Psi}^{\varrho}(t,a)\left(\delta+\sum_{a <t_k <t}\zeta_k \right)+\mathbf{I}_{a^+}^{\mu; \, \Psi}f(t,u(t))\right|\\
&\leq \Omega_{\Psi}^{\varrho}(t,a) \left(|\delta|+\sum_{k=1}^{m}|\zeta_k|\right) + \frac{1}{\Gamma(\mu)} \int_{a}^{t}\mathcal{L}_{\Psi}^{\mu}(t,\sigma)\,|f(\sigma,u(\sigma))|\,d\sigma\\
&\leq \Omega_{\Psi}^{\varrho}(t,a) \left(|\delta|+\sum_{k=1}^{m}|\zeta_k|\right)+ \frac{1}{\Gamma(\mu)} \int_{a}^{t}\mathcal{L}_{\Psi}^{\mu}(t,\sigma)\,|f(\sigma,u(\sigma))- f(\sigma,0)|\,d\sigma\\
& \qquad + \frac{1}{\Gamma(\mu)} \int_{a}^{t}\mathcal{L}_{\Psi}^{\mu}(t,\sigma)\,|f(\sigma,0)|\,d\sigma\\
&\leq \Omega_{\Psi}^{\varrho}(t,a) \left(|\delta|+\sum_{k=1}^{m}|\zeta_k|\right)+
\frac{L \, \Gamma(\varrho)}{\Gamma(\mu+\varrho)}(\Psi(t)-\Psi(a))^{1-\varrho+\mu} \|u\|_{{\mathcal{PC}}_{1-\varrho;\, \Psi}\left( \mathcal{J},\,\mathbb{R}\right)}\\
& \qquad + \frac{\mathcal{M} \, (\Psi(t)-\Psi(a))^{\mu}}{\Gamma(\mu+1)} \\
&\leq \Omega_{\Psi}^{\varrho}(t,a) \left(|\delta|+\sum_{k=1}^{m}|\zeta_k|\right)+
\frac{L \, \Gamma(\varrho)}{\Gamma(\mu+\varrho)}(\Psi(t)-\Psi(a))^{1-\varrho +\mu} \, r\\
& \qquad +  \frac{\mathcal{M}\, (\Psi(t)-\Psi(a))^{\mu}}{\Gamma(\mu+1)}.
\end{align*}

Thus,
\begin{align*}
|(\Psi(t)-\Psi(a))^{1-\varrho}\, \mathcal{T} u(t)| 
&\leq \frac{1}{\Gamma(\varrho)}\left(|\delta|+\sum_{k=1}^{m}|\zeta_k|\right)+
\frac{L \, \Gamma(\varrho)}{\Gamma(\mu+\varrho)}(\Psi(T)-\Psi(a))^{\mu} \, r\\
& \qquad + \frac{\mathcal{M} \,(\Psi(T)-\Psi(a))^{1-\varrho+\mu}}{\Gamma(\mu+1)},~ t \in \mathcal{J}.
\end{align*} 

From the choices of constants  $r$ and $L$, it can be easily verified that
\begin{align*}
 \left\|\mathcal{T} u\right\|_{{\mathcal{PC}}_{1-\varrho;\, \Psi}\left( \mathcal{J},\,\mathbb{R}\right)} \leq r.
\end{align*} 
This proves $\mathcal{T} \mathcal{B}_r\subset \mathcal{B}_r.$

Now, we prove that the operator $\mathcal{T}$ is a contraction on $\mathcal{B}_r$. Let any $u,v \in \mathcal{B}_r$. Then by assumption ($A_2$) for any $t\in \mathcal{J}$,
\begin{align*}
&\left|(\Psi(t)-\Psi(a))^{1-\varrho}(\mathcal{T}u(t)-\mathcal{T}v(t))\right|\\
& =  \left|(\Psi(t)-\Psi(a))^{1-\varrho} \left( \left\{\Omega_{\Psi}^{\varrho}(t,a)\left(\delta+\sum_{a <t_k <t}\zeta_k \right)+\mathbf{I}_{a^+}^{\mu; \, \Psi}f(t,u(t))\right \}\right.\right.\\
& \qquad-  \left. \left.\left\{\Omega_{\Psi}^{\varrho}(t,a)\left(\delta+\sum_{a <t_k <t}\zeta_k\right)+\mathbf{I}_{a^+}^{\mu; \, \Psi}f(t,v(t))\right\} \right) \right|\\ 
&= \left|(\Psi(t)-\Psi(a))^{1-\varrho}\left(\mathbf{I}_{a^+}^{\mu; \, \Psi}f(t,u(t))-\mathbf{I}_{a^+}^{\mu; \, \Psi}f(t,v(t))\right)\right|\\
&\leq \frac{(\Psi(t)-\Psi(a))^{1-\varrho}}{\Gamma(\mu)}\int_{a}^{t}\mathcal{L}_{\Psi}^{\mu}(t,\sigma)\,
\left|f(\sigma,u(\sigma))- f(\sigma,v(\sigma))\right|\,d\sigma\\
&\leq \frac{L\, (\Psi(t)-\Psi(a))^{1-\varrho}}{\Gamma(\mu)}\int_{a}^{t}\mathcal{L}_{\Psi}^{\mu}(t,\sigma)\,\left|u(\sigma)- v(\sigma)\right|\,d\sigma\\
&\leq \frac{L\, \Gamma(\varrho)}{\Gamma(\mu+\varrho)}(\Psi(t)-\Psi(a))^{\mu} \|u-v\|_{{\mathcal{PC}}_{1-\varrho;\, \Psi}\left( \mathcal{J},\,\mathbb{R}\right)}.
\end{align*}

From the choice of constant $L$, it follows that
\begin{align*}
\|\mathcal{T}u-\mathcal{T}v\|_{{\mathcal{PC}}_{1-\varrho;\, \Psi}\left( \mathcal{J},\,\mathbb{R}\right)}
&\leq \frac{1}{2} \|u-v\|_{{\mathcal{PC}}_{1-\varrho;\, \Psi}\left( \mathcal{J},\,\mathbb{R}\right)}.
\end{align*}

Thus, $\mathcal{T}$ is a contraction and by  the Banach contraction principle it  has a unique fixed point in $ \mathcal{B}_r \subseteq {\mathcal{PC}}_{1-\varrho;\, \Psi}\left( \mathcal{J},\,\mathbb{R}\right)$ which is the unique solution of  impulsive $\Psi$-HFDE \eqref{e11}-\eqref{e13}.
\end{proof}

\section{Nonlocal Impulsive $\Psi$-HFDE }

In this section we examine the existence and uniqueness results for impulsive $\Psi$-HFDE with non local initial conditions given by
\begin{align}
&^H \mathbf{D}^{\mu, \, \nu; \, \Psi}_{a^+}u(t)=f(t, u(t)),~t \in \mathcal{J}-\{t_1, t_2,\cdots ,t_m\},\label{51}\\
&\Delta \mathbf{I}_{a^+}^{1-\varrho; \, \Psi}u(t_k)= \zeta_k \in \mathbb{R},  ~~\,  ~k = 1,2,\cdots,m, \label{52}\\
&\mathbf{I}_{a^+}^{1-\varrho; \, \Psi}u(a)+ g(u)=\delta \in \mathbb{R} , \label{53}  
\end{align} 
where $\mu, \nu, \varrho$ and the function $f$ are as given in the problem 
\eqref{e11}-\eqref{e13} and $g:{\mathcal{PC}}_{1-\varrho;\, \Psi}\left( \mathcal{J},\,\mathbb{R}\right) \to \mathbb{R} $ is a continuous function.

\begin{theorem}$(\mathbf{Existence})$
Assume that the function $f:(a,T]\times \mathbb{R} \to \mathbb{R}$ is continuous and satisfies the conditions $(A_1)-(A_2)$. Further, assume that $g:{\mathcal{PC}}_{1-\varrho;\, \Psi}\left( \mathcal{J},\,\mathbb{R}\right) \to \mathbb{R} $ is a continuous function that satisfy:
\begin{itemize}
\item[($A_3$)] $\left|g(u)-g(v)\right|\leq L_g \|u-v\|_{{\mathcal{PC}}_{1-\varrho;\, \Psi}\left( \mathcal{J},\,\mathbb{R}\right)}, ~ u,v \in {\mathcal{PC}}_{1-\varrho;\, \Psi}\left( \mathcal{J},\,\mathbb{R}\right),$ with $0< L_g \leq \dfrac{1}{6}\, \Gamma(\varrho).$
 \end{itemize}
 
Then, the nonlocal impulsive $\Psi$-HFDE \eqref{51}-\eqref{53} has at least one solution in ${\mathcal{PC}}_{1-\varrho;\, \Psi}\left( \mathcal{J},\,\mathbb{R}\right)$.
\end{theorem}
\begin{proof} By applying the Lemma \ref{fie},  the equivalent fractional integral equation of the nonlocal impulsive $\Psi$-HFDE \eqref{51}-\eqref{53} is given as follows
\begin{equation}\label{ie}
 u(t) =\Omega_{\Psi}^{\varrho}(t,a)\,\left(\delta-g(u)+\sum_{a<t_k<t}\zeta_k\right)+\mathbf{I}_{a^+}^{\mu; \, \Psi}f(t,u(t)), ~t \in \mathcal{J}.
\end{equation}

Consider the set 
$$
 \mathcal{B}_{r^{*}} = \left\{u\in \mathcal{PC}_{1-\varrho;\, \Psi}\left( \mathcal {J},\,\mathbb{R}\right) : \mathbf{I}_{a^+}^{1-\varrho; \, \Psi}u(a)+g(u)=\delta, ~\|u\|_{{\mathcal{PC}}_{1-\varrho;\, \Psi}\left( \mathcal{J},\,\mathbb{R}\right)} \leq {r^{*}} \right\},
 $$ 
where
$$
{r^{*}} \geq 3
\left(\frac{1}{\Gamma(\varrho)}\left\{|\delta|+G+\sum_{k=1}^{m}|\zeta_k|\right\} + \frac{\mathcal{M}}{\Gamma(\mu+1)}\,(\Psi(T)-\Psi(a))^{1-\varrho+\mu} \right),
$$
$G= |g(0)|$ and 
$ \mathcal{M} = \sup_{\sigma\in \mathcal{J}} |f(\sigma,0)|$.

Define operator $\mathcal{R}$ and  $\mathcal{Q} ^*$ on $\mathcal{B}_{r^{*}}$ by
\begin{align*}
&\mathcal{R}u(t)=\Omega_{\Psi}^{\varrho}(t,a)\,\left(\delta-g(u)+\sum_{a<t_k<t}\zeta_k\right) ,\, t \in \mathcal{J},\\
&\mathcal{Q}^*u (t) =  \mathbf{I}_{a^+}^{\mu; \, \Psi}f(t,u(t)) ,\, t \in \mathcal{J}.
\end{align*} 

Then the fractional integral equation \eqref{ie} is equivalent to the  operator equation 
\begin{align} \label{ope}
u=\mathcal{R}u + \mathcal{Q}^*u,\,  u\in {\mathcal{PC}}_{1-\varrho;\, \Psi}\left( \mathcal{J},\,\mathbb{R}\right). 
\end{align}

We apply the  Krasnoselskii's fixed point theorem
(Theorem \ref{kf})  to prove that the operator equation \eqref{ope} has fixed point. Firstly, we show that $\mathcal{R}u+\mathcal{Q}^*v \in \mathcal{B}_{r^{*}}$ for any $u,v\in \mathcal{B}_{r^{*}}$. By assumption ($A_2$) and ($A_3$), for any $u,v \in \mathcal{B}_{r^{*}}$ and $t\in \mathcal{J}$, 
\begin{align*}
&\left|(\Psi(t)-\Psi(a))^{1-\varrho}(\mathcal {R}u(t)+\mathcal {Q}^*v(t))\right|\\
& =  \left|(\Psi(t)-\Psi(a))^{1-\varrho}\left\{\Omega_{\Psi}^{\varrho}(t,a)\,\left(\delta- g(u)+\sum_{a<t_k<t}\zeta_k\right)+\mathbf{I}_{a^+}^{\mu; \, \Psi}f(t,v(t))\right \}\right|\\
&\leq \frac{1}{\Gamma(\varrho)}\left(|\delta|+|g(u)|+\sum_{k=1}^{m}|\zeta_k|\right) + \frac{(\Psi(t)-\Psi(a))^{1-\varrho}}{\Gamma(\mu)} \int_{a}^{t}\mathcal{L}_{\Psi}^{\mu}(t,\sigma)\,|f(\sigma,v(\sigma))|\,d\sigma\\
&\leq \frac{1}{\Gamma(\varrho)}\left(|\delta|+|g(u)-g(0)|+|g(0)|+\sum_{k=1}^{m}|\zeta_k|\right)\\
&\qquad + \frac{(\Psi(t)-\Psi(a))^{1-\varrho}}{\Gamma(\mu)} \int_{a}^{t}\mathcal{L}_{\Psi}^{\mu}(t,\sigma)\,|f(\sigma,v(\sigma))- f(\sigma,0)|\,d\sigma\\
&\qquad  + \frac{(\Psi(t)-\Psi(a))^{1-\varrho}}{\Gamma(\mu)} \int_{a}^{t}\mathcal{L}_{\Psi}^{\mu}(t,\sigma)\,|f(\sigma,0)|\,d\sigma\\
&\leq \frac{1}{\Gamma(\varrho)}\left(|\delta|+L_g \|u\|_{{\mathcal{PC}}_{1-\varrho;\, \Psi}\left( \mathcal{J},\,\mathbb{R}\right)}+ G +\sum_{k=1}^{m}|\zeta_k|\right) \\
&\qquad + \frac{L \,\Gamma(\varrho)}{\Gamma(\mu+\varrho)}(\Psi(t)-\Psi(a))^{\mu} \|v\|_{{\mathcal{PC}}_{1-\varrho;\, \Psi}\left( \mathcal{J},\,\mathbb{R}\right)}+ \frac{\mathcal{M} \,(\Psi(t)-\Psi(a))^{1-\varrho+\mu}}{\Gamma(\mu+1)}\\
&\leq \frac{1}{\Gamma(\varrho)}\left(|\delta|+ G +\sum_{k=1}^{m}|\zeta_k|\right) + \frac{L_g}{\Gamma(\varrho)}r^{*}\\
&\qquad +\frac{L\Gamma(\varrho)}{\Gamma(\mu+\varrho)}(\Psi(T)-\Psi(a))^{\mu}r^{*}+\frac{\mathcal{M}(\Psi(T)-\Psi(a))^{1-\varrho+\mu}}{\Gamma(\mu+1)}.
\end{align*}

From the choice of $r^{*}$, $L$ and $L_g$, from the above inequality, we obtain 
\begin{align*}
 \left\|(\mathcal {R}u+\mathcal {Q}^*v)\right\|_{{\mathcal{PC}}_{1-\varrho;\, \Psi}\left( \mathcal{J},\,\mathbb{R}\right)} \leq r^{*}.
\end{align*}

Further, one can  verify that
\begin{align*}
 \mathbf{I}_{a^+}^{1-\varrho; \, \Psi}(\mathcal {R}u+\mathcal {Q}^*v)(a) + g(u)=\delta .
\end{align*}
This shows that $\mathcal{R}u+\mathcal{Q}^*v \in \mathcal{B}_{r^{*}}.$

Next, we prove that $\mathcal{R}$ is a contraction mapping. Let any $u,v \in \mathcal{B}_{r^{*}}$ and $t \in \mathcal{J}$.

Consider
\begin{align*}
&\left|(\Psi(t)-\Psi(a))^{1-\varrho}(\mathcal {R}u(t)-\mathcal {R}v(t))\right|\\
&=\left|(\Psi(t)-\Psi(a))^{1-\varrho}\left\{\Omega_{\Psi}^{\varrho}(t,a)\,\left(\delta- g(u)+\sum_{a<t_k<t}\zeta_k\right)\right.\right.\\
&\qquad \left.\left. -\, \Omega_{\Psi}^{\varrho}(t,a)\,\left(\delta-g(v)+\sum_{a<t_k<t}\zeta_k\right)\right \} \right|\\
&=\frac{1}{\Gamma(\varrho)}\left| g(u)-g(v)\right|\\
&\leq \frac{L_g}{\Gamma(\varrho)}\,  \|u-v\|_{{\mathcal{PC}}_{1-\varrho;\, \Psi}\left( \mathcal{J},\,\mathbb{R}\right)}.
\end{align*}

From the choice of $L_g$, we obtain
\begin{align*}
\|Ru-Rv\|_{{\mathcal{PC}}_{1-\varrho;\, \Psi}\left( \mathcal{J},\,\mathbb{R}\right)}
\leq \frac{1}{6} \, \|u-v\|_{{\mathcal{PC}}_{1-\varrho;\, \Psi}\left( \mathcal{J},\,\mathbb{R}\right)}.
\end{align*}

This shows that $\mathcal {R}$ is a contraction. The operator $\mathcal {Q}^*$  is compact and continuous as proved in the Theorem \ref{ex}. Hence by ${\mathcal{PC}}_{1-\varrho;\, \Psi}\left( \mathcal{J},\,\mathbb{R}\right)$-type Arzela-Ascoli Theorem \ref{pc}  $\mathcal {Q}^*$ is compact on $\mathcal{B}_{r^{*}}$. Further, as discussed in the 
 proof of  Theorem \ref{ex} the non local impulsive $\Psi$-HFDE \eqref{51}-\eqref{53} has at least one solution in ${\mathcal{PC}}_{1-\varrho;\, \Psi}\left( \mathcal{J},\,\mathbb{R}\right)$.
\end{proof}

\begin{theorem}\textbf{(Uniqueness)} \label{unique2}
Assume that the function $f:(a,T]\times \mathbb{R} \to \mathbb{R}$ is continuous and satisfies the conditions $(A_1)-(A_3)$. Then, non local impulsive $\Psi$--HFDE  \eqref{51}-\eqref{53} has a unique solution in the weighted space ${\mathcal{PC}}_{1-\varrho;\, \Psi}\left( \mathcal{J},\,\mathbb{R}\right)$.
\end{theorem}

\begin{proof}
The proof can be completed following the same steps as in the proof the Theorem \ref{unique1}. 
\end{proof} 
\section{Applications} 
By  taking 
$$
\Psi(t)=t ~\mbox{and} ~~\nu \to 1.
$$
the impulsive $\Psi$-HFDE \eqref{e11}-\eqref{e13} reduces to Caputo impulsive FDE of the form:
\begin{align}
& ^CD^{\mu}_{0^+}u(t)=f(t, u(t)),~t \in \mathcal{J}=[a,T]-\{t_1, t_2,\cdots ,t_m\},\label{ap11}\\
&\Delta u(t_k)= \zeta_k \in \mathbb{R},  ~k = 1,2,\cdots,m, \label{ap12}\\
& u(a)=\delta \in \mathbb{R}, \label{ap13}
\end{align}
and we have the following existence and uniqueness theorems for Caputo impulsive FDE \eqref{ap11}-\eqref{ap13} as an applications of the Theorem \ref{ex} and Theorem \ref{unique1}.

\begin{theorem}\label{excaputo}
Assume that the function $f\in C(\mathcal{J}, \, \mathbb{R})$  satisfies the Lipschitz  condition
$$
|f(t,u)-f(t,v)|\leq L |u-v|, ~ t\in \mathcal{J}, ~u,v \in \mathbb{R}$$ with $0<L\leq \dfrac{\Gamma(\mu+1)}{2(T-a)^\mu}$.  Then, the Caputo impulsive FDE \eqref{ap11}-\eqref{ap13} has at least one solution in the  space ${\mathcal{PC}}\left( \mathcal{J},\,\mathbb{R}\right)$.
\end{theorem}

\begin{theorem} \label{uniquecaputo}
Under the suppositions of the Theorem \ref{excaputo} the impulsive Caputo FDE \eqref{e11}-\eqref{e13} has a unique solution in the space ${\mathcal{PC}}\left( \mathcal{J},\,\mathbb{R}\right)$.
\end{theorem}

\section{Examples}
In this section, we give examples to illustrate the utility of the results we obtained. 

\begin{ex}
Consider, the impulsive $\Psi$-HFDE 
\begin{align}
&^H \mathbf{D}^{\mu, \, \nu; \, \Psi}_{0^+}u(t)=\frac{9}{5\Gamma(\frac{2}{3})}(\Psi(t)-\Psi(0))^{\frac{5}{3}}-\frac{(\Psi(t)-\Psi(0))^4}{16}+\frac{1}{16}u^2, t\in [0,1]-\left\{\frac{1}{2}\right\},\label{p1}\\
&\mathbf{I}_{0^+}^{1-\varrho; \, \Psi}u(0)=0, \label{p2} \\
&\Delta \mathbf{I}_{0^+}^{1-\varrho; \, \Psi}u\left( \frac{1}{2}\right) = \sigma\in \mathbb{R},\label{p3} 
\end{align}
$0<\mu<1,~0\leq\nu\leq 1, ~\varrho=\mu+\nu-\mu\nu$ and $\Psi:[0,1] \to \mathbb{R}$ is as defined in preliminaries.

Define the function $f:[0,1]\times\mathbb{R} \to \mathbb{R}$ by  
$$
f(t,u)=\frac{9}{5\Gamma(\frac{2}{3})}(\Psi(t)-\Psi(0))^{\frac{5}{3}}-\frac{(\Psi(t)-\Psi(0))^4}{16}+\frac{1}{16}u^2.
$$

Clearly, $$|f(t,u)-f(t,v)|\leq\frac{1}{8}|u-v|, u, v \in \mathbb{R}, ~ t \in [0,1].$$
Thus $f$ satisfies the Lipschitz condition with the constant $L=\frac{1}{8}.$  
If the function $\Psi$ satisfies the condition 
\begin{align}\label{L}
L\leq \frac{\Gamma(\mu+\varrho)}{2\Gamma(\varrho)(\Psi(1)-\Psi(0))^\mu}
\end{align}
then problem \eqref{p1}-\eqref{p3} has unique solution.

For instance, consider the  particular case of the problem \eqref{p1}-\eqref{p3}. By  taking $$\Psi(t)=t, ~\mu=\frac{1}{3} ~\mbox{and} ~~\nu \to 1.$$ 

Then the problem \eqref{p1}-\eqref{p3} reduces to impulsive FDE involving Caputo fractional derivative operator of the form:
\begin{align}
&^CD^{\frac{1}{3}}_{0^+}u(t)=\frac{9}{5\Gamma(\frac{2}{3})}t^{\frac{5}{3}}-\frac{t^4}{16}+\frac{1}{16}u^2, t\in [0,1]-\left\{\frac{1}{2}\right\} \label{p4}\\
&\Delta u\left( \frac{1}{2}\right) = 0\label{p5}\\
&u(0)=0. \label{p6}
\end{align}

Note that
\begin{align*}
\frac{\Gamma(\mu+\varrho)}{2\Gamma(\varrho)(\Psi(T)-\Psi(a))^\mu}= \frac{1}{2} \, \Gamma \left( \frac{4}{3}\right) \approx 0.445.
\end{align*}

Since $L =\frac{1}{8}$, the condition \eqref{L} is satisfied. Using the Theorem \ref{ex}  with $\Psi(t)=t,~a=0, ~T=1, ~\mu=\frac{1}{3} ~\mbox{and} ~~\nu \to 1$ the problem \eqref{p4}-\eqref{p6} has a  solution on $[0,1]$. 

By direct substitution one can verify that $ u(t)=t^2$ is the solution of the problem \eqref{p4} - \eqref{p6}.
\end{ex}

\begin{ex}
Consider an impulsive $\Psi$-HFDE
\begin{align}
&^H \mathbf{D}^{\mu, \, \nu; \, \Psi}_{0^+}u(t)=\frac{\sin^4(\Psi(t)-\Psi(0))}{((\Psi(t)-\Psi(0))+3)^3} \frac{|u(t)|}{1+|u(t)|}, t\in [0,1]-\left\{\frac{1}{3}\right\},\label{p7}\\
&\Delta \mathbf{I}_{0^+}^{1-\varrho; \, \Psi}u\left(\frac{1}{3}\right)= \sigma ,\label{p8} \\
&\mathbf{I}_{0^+}^{1-\varrho; \, \Psi}u(0)=\delta,\label{p9}
\end{align}
where $0<\mu<1,~0\leq\nu\leq 1, ~\varrho=\mu+\nu-\mu\nu$.

Define the function $f:[0,1]\times\mathbb{R} \to \mathbb{R}$ by   $$f(t,u)= \frac{\sin^4(\Psi(t)-\Psi(0))}{((\Psi(t)-\Psi(0))+3)^3} \frac{|u|}{1+|u|}.$$

Let $u,v \in \mathbb{R}  ~\mbox{and} ~~t\in [0,1]$.  Then,
\begin{align*}
|f(t,u)- f(t,v)|&= \left|\frac{\sin^4(\Psi(t)-\Psi(0))}{((\Psi(t)-\Psi(0))+3)^3} \frac{|u|}{1+|u|}-\frac{\sin^4(\Psi(t)-\Psi(0))}{((\Psi(t)-\Psi(0))+3)^3} \frac{|v|}{1+|v|}\right|\\
& \leq \frac{1}{\left( (\Psi(t)-\Psi(0))+3\right) ^3} \left| \frac{|u|}{1+|u|}- \frac{|v|}{1+|v|} \right|\\
&\leq  \frac{1}{\left( (\Psi(1)-\Psi(0))+3\right) ^3}|u-v|.
\end{align*}

This proves $f$ is Lipschitz function with the constant 
$$ L=\frac{1}{\left( (\Psi(1)-\Psi(0))+3\right) ^3}.$$  

By Theorem \ref{ex} the problem \eqref{p7}-\eqref{p9} has a  solution if 
$$\frac{1}{\left( (\Psi(1)-\Psi(0))+3\right) ^3}\leq \frac{\Gamma(\mu+\varrho)}{2\Gamma(\varrho)(\Psi(1)-\Psi(0))^\mu}.$$
\end{ex}
\section*{Concluding remarks}

We close the present paper with the destinations we accomplished. We investigated the existence and uniqueness of solutions of nonlinear $\Psi$-HFDE and of also their respective extension to nonlocal case by means of strong analysis results. Some examples were illustrated in order to elucidate the results obtained. It is noted that since the $\Psi$-Hilfer fractional derivative is global and contains a wide class of fractional derivatives, the properties investigated herein are also valid for their respective particular cases. 

Here we have not investigated the continuous dependence on the various data and Ulam-Hyers stabilities of solution of (\ref{e11})-(\ref{e13}), which is the point of  our next investigation and will be published a future work.

Now, if we consider $\Psi(t)=t$ in the problem (\ref{e11})-(\ref{e13}) with $A: D(A)\subset X \rightarrow X$ generator of $C_{0}$-semigroup ($\mathbb{P}_{t\geq 0}$) on a Banach space $X$, we have the following impulsive $\Psi$-HFDE with initial condition
\begin{align}
& ^H \mathbf{D}^{\mu,\, \nu}_{a^+}u(t)=Au(t)+f(t, u(t)),~t \in \mathcal{J}=[a,T]-\{t_1, t_2,\cdots ,t_m\}, \label{r1}\\
&\Delta \mathbf{I}_{a^+}^{1-\varrho}u(t_k)= \zeta_k \in \mathbb{R},  ~k = 1,2,\cdots,m, \\
& \mathbf{I}_{a^+}^{1-\varrho; \, \Psi}u(a)=\delta \in \mathbb{R}\label{r3}, 
\end{align}
with the same conditions of problem (\ref{e11})-(\ref{e13}). The next step of the research is to analyze the problem (\ref{r1})-(\ref{r3}). But the question may arise `` Why not get the existence and uniqueness of mild solutions to the problem (\ref{r1})-(\ref{r3}), with a formulation in the sense?". The reason for non-investigation with the $\Psi$-Hilfer fractional derivative comes from the fact of the non-existence of an integral transform, in particular, of Laplace with respect to another function, since it is an important condition in the investigation of the mild solution. Research in this sense has been developed and, in the near future, results can be published.

\section*{Acknowledgment}
The third author of this paper is financially supported by the PNPD-CAPES scholarship of the Pos-Graduate Program in Applied Mathematics IMECC-Unicamp.

\end{document}